\documentclass[11pt]{article}
\usepackage{amsmath,amssymb,amsfonts,amsthm}

\theoremstyle{plain}
\newtheorem{theorem}{Theorem}[section]
\newtheorem{proposition}[theorem]{Proposition}
\newtheorem{lemma}[theorem]{Lemma}
\theoremstyle{definition}
\newtheorem{definition}[theorem]{Definition}
\theoremstyle{remark}
\newtheorem{remark}[theorem]{Remark}

\newcommand{\tr}{\mathop{\mathrm{tr}}}
\newcommand{\adj}{\mathop{\mathrm{adj}}}

\begin{document}

\title{%
Determination of accessory parameters
in a system of the Okubo normal form}

\author{Toshiaki Yokoyama\thanks{%
  Faculty of Engineering,
  Chiba Institute of Technology,
  Narashino,
  Chiba 275-0023, 
  Japan
  }}

\date{}

\maketitle

\begin{abstract}
A system of differential equations
of the Okubo normal form 
containing accessory parameters is considered.
A condition for determining 
special values of the accessory 
parameters is given. It is 
shown that the special values 
give the differential equation satisfied 
by a product of the
Gauss hypergeometric functions.

\medskip


\medskip


\end{abstract}

\section{Introduction}
We are concerned with
the system of differential equations satisfied by
a product of the Gauss hypergeometric functions.
Set
\begin{equation*}
w=\begin{pmatrix}
 f_{1}f_{2} \\ x{f_{1}}'f_{2} \\ x f_{1}{f_{2}}' \\ x^2{f_{1}}'{f_{2}}'
 \end{pmatrix},
\quad
f_{j}={}_{2}F_{1}(\alpha_{j},\beta_{j},\gamma_{j};x),
\quad
{f_{j}}'=\dfrac{df_{j}}{dx}
\quad
(j=1,2).
\end{equation*}
The function $f_{j}$ satisfies the differential equation
\begin{equation*}
{f_{j}}''-p_{j}(x){f_{j}}'-q_{j}(x)f_{j}=0,
\end{equation*}
where
\begin{equation*}
p_{j}(x)=-\frac{\gamma_{j}}{x}+\frac{\gamma_{j}-1-\alpha_{j}-\beta_{j}}{x-1},
\quad
q_{j}(x)=-\frac{\alpha_{j}\beta_{j}}{x(x-1)},
\end{equation*}
and hence $w$ satisfies the system of differential equations
\begin{equation*}
\frac{dw}{dx}
=
\begin{pmatrix}
 0 & \frac{1}{x} & \frac{1}{x} & 0 \\[\jot]
 x q_{1}(x) & \frac{1}{x}+p_{1}(x) & 0 & \frac{1}{x} \\[\jot]
 x q_{2}(x) & 0 & \frac{1}{x}+p_{2}(x) & \frac{1}{x} \\[\jot]
 0 & x q_{2}(x) & x q_{1}(x) & \frac{2}{x}+p_{1}(x)+p_{2}(x) \\
\end{pmatrix}
w,
\end{equation*}
which is a Fuchsian system of normal form,
namely,
\begin{equation}
\frac{dw}{dx}
=
\left(
\frac{1}{x}H_{0}
+
\frac{1}{x-1}H_{1}
\right)w
\label{sys:prod}
\end{equation}
with
\begin{align*}
H_{0}
&=
\begin{pmatrix}
0 & 1 & 1 & 0 \\
0 & 1-\gamma_{1} & 0 & 1 \\
0 & 0 & 1-\gamma_{2} & 1 \\
0 & 0 & 0 & 2-\gamma_{1}-\gamma_{2} \\
\end{pmatrix},
\\
H_{1}
&=
\begin{pmatrix}
0 & 0 & 0 & 0 \\
-\alpha_{1}\beta_{1} & \gamma_{1}-1-\alpha_{1}-\beta_{1} & 0 & 0  \\
-\alpha_{2}\beta_{2} & 0 
  & \hbox to 35pt{\hss$\gamma_{2}-1-\alpha_{2}-\beta_{2}$\hss} & 0  \\
0 & -\alpha_{2}\beta_{2} & -\alpha_{1}\beta_{1}
  & \gamma_{1}+\gamma_{2}-2-\alpha_{1}-\alpha_{2}-\beta_{1}-\beta_{2} \\
\end{pmatrix}.
\end{align*}
We denote by $I_{m}$ the $m\times m$ identity matrix.
Since 
the characteristic polynomial $\varphi(t)=\det(tI_{4}+H_{0}+H_{1})$
of the residue matrix $-(H_{0}+H_{1})$ at $x=\infty$
is factored into
\begin{equation*}
\varphi(t)
=(t-\alpha_{1}-\alpha_{2})
 (t-\beta_{1}-\beta_{2})
 (t-\alpha_{1}-\beta_{2})
 (t-\beta_{1}-\alpha_{2}),
\end{equation*}
the Riemann scheme of \eqref{sys:prod} is
\begin{equation*}
\left\{ \ \begin{matrix}
  x=0 & x=1 & x=\infty  \\
  0            & 0                                 & \alpha_{1}+\alpha_{2}  \\
  1-\gamma_{1} & \gamma_{1}-1-\alpha_{1}-\beta_{1} & \beta_{1}+\beta_{2}  \\
  1-\gamma_{2} & \gamma_{2}-1-\alpha_{2}-\beta_{2} & \alpha_{1}+\beta_{2}  \\
  2-\gamma_{1}-\gamma_{2} 
   & \ \gamma_{1}+\gamma_{2}-2-\alpha_{1}-\alpha_{2}-\beta_{1}-\beta_{2} \
                                                   & \beta_{1}+\alpha_{2}  \\
   \end{matrix} \ \right\}
\end{equation*}
and the spectral type of \eqref{sys:prod} is $((1111),(1111),(1111))$
if the parameters 
$\alpha_{j}$, $\beta_{j}$, $\gamma_{j}$ ($j=1,2$)
are generic.

In this paper we consider a case of repeated local exponents.
We suppose that
\begin{equation*}
1-\gamma_{2}=1-\gamma_{1}
\quad \mbox{and} \quad
\gamma_{1}+\gamma_{2}-2-\alpha_{1}-\alpha_{2}-\beta_{1}-\beta_{2}=0, 
\end{equation*}
or equivalently
\begin{equation}
\gamma_{j}=\frac{1}{2}(\alpha_{1}+\alpha_{2}+\beta_{1}+\beta_{2})+1
\quad (j=1,2).
\label{cond:Okubo}
\end{equation}
In this case the system \eqref{sys:prod},
the spectral type of which is $((211),(211),(1111))$,
can be transformed into what is called a system of the Okubo normal form.  
We denote by $\tilde{H}_{0}$, $\tilde{H}_{1}$
the residue matrices $H_{0}$, $H_{1}$ with $\gamma_{j}$ ($j=1,2$)
replaced by the right member
of \eqref{cond:Okubo}.
Introducing the notation 
\begin{equation*}
\lambda_{\iota_{1}\iota_{2}\iota_{3}\iota_{4}}
=\frac{1}{2}\{(\iota_{1}\alpha_{1})+(\iota_{2}\alpha_{2})
             +(\iota_{3}\beta_{1}) +(\iota_{4}\beta_{2})\}
, \quad \iota_{j}={+}, \, {-} \ \mbox{or} \ 0,
\end{equation*}
we can write the matrices $\tilde{H}_{0}$, $\tilde{H}_{1}$ 
in the form
\begin{align*}
\tilde{H}_{0}
&=
\lambda_{----}I_{4}
+
\begin{pmatrix}
\lambda_{++++} & 1 & 1 & 0 \\
0 & 0 & 0 & 1 \\
0 & 0 & 0 & 1 \\
0 & 0 & 0 & \lambda_{----} \\
\end{pmatrix},
\\
\tilde{H}_{1}
&=
\begin{pmatrix}
0 & 0 & 0 & 0 \\
-4\lambda_{+000}\lambda_{00+0}  &  \lambda_{-+-+}  &  0  &  0  \\
-4\lambda_{0+00}\lambda_{000+}  &  0  &  \lambda_{+-+-}  &  0  \\
0 & -4\lambda_{0+00}\lambda_{000+} & -4\lambda_{+000}\lambda_{00+0} & 0 \\
\end{pmatrix}.
\end{align*}
Here note that
$\lambda_{\bar{\iota}_{1}\bar{\iota}_{2}\bar{\iota}_{3}\bar{\iota}_{4}}
 =-\lambda_{\iota_{1}\iota_{2}\iota_{3}\iota_{4}}$
holds for $\bar{\iota}_{j}=-\iota_{j}$.
We change the variable $w$ to $u$
by
\begin{equation}
w=x^{\lambda_{----}}Pu,
\quad
P=
\begin{pmatrix}
1 & 1 & 0 & 0\\
0 & \lambda_{----} & \lambda_{+-+-} & 0             \\
0 & \lambda_{----} & 0             & \lambda_{-+-+} \\
0 & {\lambda_{----}}^2 
  & 4\lambda_{0+00}\lambda_{000+} & 4\lambda_{+000}\lambda_{00+0}
\end{pmatrix}
\label{trans:w2u}
\end{equation}
under the condition 
$\det P=\lambda_{----}\lambda_{++--}\lambda_{+-+-}\lambda_{+--+}\ne0$.
The inverse matrix of $P$ is
\begin{equation*}
P^{-1}
=
\frac{1}{\lambda_{++--}\lambda_{+--+}}
\begin{pmatrix}
\lambda_{++--}\lambda_{+--+} 
  & \frac{4\lambda_{0+00}\lambda_{000+}}{\lambda_{++++}}
  & \frac{4\lambda_{+000}\lambda_{00+0}}{\lambda_{----}}
  & \frac{\lambda_{+-+-}}{\lambda_{----}}
\\[\jot]
0 & \frac{4\lambda_{0+00}\lambda_{000+}}{\lambda_{----}}
  & \frac{4\lambda_{+000}\lambda_{00+0}}{\lambda_{++++}}
  & \frac{\lambda_{-+-+}}{\lambda_{----}}
\\[\jot]
0 & \frac{\lambda_{-+++}\lambda_{++-+}}{\lambda_{-+-+}}
  & \frac{4\lambda_{+000}\lambda_{00+0}}{\lambda_{+-+-}}
  & 1
\\[\jot]
0 & \frac{4\lambda_{0+00}\lambda_{000+}}{\lambda_{+-+-}}
  & \frac{\lambda_{+-++}\lambda_{+++-}}{\lambda_{-+-+}}
  & -1
\end{pmatrix}
\end{equation*}
that consists of
left eigenvectors of $\tilde{H}_{1}$ and $\tilde{H}_{0}-\lambda_{----}I_{4}$
with respect to the eigenvalue $0$.
Then we have
\begin{align*}
P^{-1}\bigl(\tilde{H}_{0}-\lambda_{----}I_{4}\bigr)P
&=
\begin{pmatrix}
  \lambda_{++++} & 0
   & \frac{\lambda_{+-++}\lambda_{+++-}}{\lambda_{++++}}
   & \frac{\lambda_{-+++}\lambda_{++-+}}{\lambda_{++++}}
  \\[\jot]
  0 & \lambda_{----}
   & \frac{4\lambda_{0+00}\lambda_{000+}}{\lambda_{----}}
   & \frac{4\lambda_{+000}\lambda_{00+0}}{\lambda_{----}}
  \\
  0 & 0 & 0 & 0
  \\
  0 & 0 & 0 & 0
\end{pmatrix},
\\
P^{-1}\tilde{H}_{1}P
&=
\begin{pmatrix}
  0 & 0 & 0 & 0
  \\
  0 & 0 & 0 & 0
  \\
  \frac{4\lambda_{+000}\lambda_{00+0}}{\lambda_{-+-+}}
  & \frac{\lambda_{-+++}\lambda_{++-+}}{\lambda_{-+-+}}
  & \lambda_{-+-+} & 0
  \\[\jot]
  \frac{4\lambda_{0+00}\lambda_{000+}}{\lambda_{+-+-}}
  & \frac{\lambda_{+-++}\lambda_{+++-}}{\lambda_{+-+-}}
  & 0 & \lambda_{+-+-}
\end{pmatrix}.
\end{align*}
Thus the resulting system 
\begin{equation*}
\frac{du}{dx}
=
\left(
\frac{1}{x}P^{-1}\bigl(\tilde{H}_{0}-\lambda_{----}I_{4}\bigr)P
+
\frac{1}{x-1}P^{-1}\tilde{H}_{1}P
\right)u
\end{equation*}
by the transformation \eqref{trans:w2u}
is written in the Okubo normal form 
\begin{equation}
\left(
x I_{4}
-
\begin{pmatrix}
  0I_{2} &       \\
    & 1I_{2}     \\
\end{pmatrix}
\right)
\frac{du}{dx}
=
A_{0} u,
\label{sys:zero}
\end{equation}
where the coefficient matrix $A_{0}$ is given by
\begin{align*}
A_{0}
&=
P^{-1}\bigl(\tilde{H}_{0}-\lambda_{----}I_{4}\bigr)P
+
P^{-1}\tilde{H}_{1}P
\\
&=
\begin{pmatrix}
\lambda_{++++} J &
   \begin{matrix}
     \frac{\lambda_{+-++}\lambda_{+++-}}{\lambda_{++++}}
   & \frac{\lambda_{-+++}\lambda_{++-+}}{\lambda_{++++}}
   \\[\jot]
     \frac{4\lambda_{0+00}\lambda_{000+}}{\lambda_{----}}
   & \frac{4\lambda_{+000}\lambda_{00+0}}{\lambda_{----}}
   \end{matrix}
\\
\begin{matrix}
    \frac{4\lambda_{+000}\lambda_{00+0}}{\lambda_{-+-+}}
  & \frac{\lambda_{-+++}\lambda_{++-+}}{\lambda_{-+-+}}
  \\[\jot]
    \frac{4\lambda_{0+00}\lambda_{000+}}{\lambda_{+-+-}}
  & \frac{\lambda_{+-++}\lambda_{+++-}}{\lambda_{+-+-}}
\end{matrix}
 & \lambda_{-+-+} J
\end{pmatrix}
\end{align*}
with
\begin{equation*}
J=\begin{pmatrix} 1 & \\ & -1 \end{pmatrix}.
\end{equation*}
Note that $A_{0}$
is similar to 
a diagonal matrix of the form
\begin{equation*}
\begin{pmatrix}
\lambda_{++--} J &                \\
              & \lambda_{+--+} J  \\
\end{pmatrix}.
\end{equation*}

Consider a size four system of the Okubo normal form
\begin{equation}
(xI_{4}-T)\frac{du}{dx}
=Au,
\label{sys:gene}
\end{equation}
where $T$ and $A$ are $4\times4$ matrices of the form
\begin{equation*}
T=
\begin{pmatrix}
  0I_{2} &       \\
    & 1I_{2}     \\
   \end{pmatrix}
\quad\mbox{and}\quad
A=
\begin{pmatrix}
  aJ     & A_{12} \\
  A_{21} & bJ     \\
\end{pmatrix},
\end{equation*}
and $A$ is assumed to be similar to
a diagonal matrix of the form
\begin{equation*}
\begin{pmatrix}
  cJ &    \\
     & dJ \\
\end{pmatrix}.
\end{equation*}
Throughout this paper,
we assume 
the condition
\begin{equation}
a, b, c, d, 2a, 2b, 2c, 2d,
a\pm b, a\pm c, a\pm d, b\pm c, b\pm d, c\pm d \not\in \mathbb{Z}.
\label{cond:base}
\end{equation}
The system \eqref{sys:gene}
has the same local exponents as \eqref{sys:zero} if 
$a=\lambda_{++++}$,  
$b=\lambda_{-+-+}$,  
$c=\lambda_{++--}$,  
$d=\lambda_{+--+}$,  
while contains two accessory parameters
as a Fuchsian system of normal form
since its spectral type is $((211),(211),(1111))$
and then its index of rigidity is
\begin{equation*}
\iota=(1-2)\cdot4^2+\bigl(\bigl(2^2+1^2+1^2\bigr)
                         +\bigl(2^2+1^2+1^2\bigr)
                         +\bigl(1^2+1^2+1^2+1^2\bigr)\bigr)=0
\end{equation*}
(see Haraoka \cite[7.4.2]{Har20} 
for the relation of the index of rigidity 
and the number of accessory parameters,
see also Proposition \ref{prop:expr_a_kai} and
Remark \ref{rem:no_AP}).  
The system \eqref{sys:zero} is nothing but \eqref{sys:gene}
with special values of the accessory parameters.

The system \eqref{sys:gene} has convergent power series solutions
near each singularity. 
Coefficient vectors of the series solution
satisfy a system of linear difference equations.
In this paper
we propose a condition 
about the systems of linear difference equations
satisfied by the coefficient vectors
for determining values of accessory parameters.

Recently Ebisu \cite{Ebi18} 
(see also Ebisu et~al.~\cite{EHKOSY19})
developed the theory of invariants
of scalar linear difference equations of higher order,
and defined the notion termed \emph{essentially the same}
for the difference equations.
Moreover, he expands the notion
to systems of linear difference equations of the first order.

\begin{definition}[Ebisu \cite{Ebi19}]
\label{def:essen_same}
Let $B(z)$ and $C(z)$ be $n\times n$ matrices that consist
of rational functions of $z$.
Two systems of difference equations
\begin{equation*}
f(z+1)=B(z)f(z) \quad \mbox{and} \quad h(z+1)=C(z)h(z),
\end{equation*}
where $f(z)$ and $h(z)$ are unknown $n$-vectors,
are said to be \emph{essentially the same} 
if there exists a diagonal transformation of the form
\begin{equation}
f(z)
=
\begin{pmatrix}
  \Gamma_{1}(z) & & & \\
  & \Gamma_{2}(z) & & \\
  & & \ddots & \\
  & & & \Gamma_{n}(z) \\
\end{pmatrix}
h(z),
\label{trans:inDef}
\end{equation}
where $\Gamma_{j}(z)$ satisfies a linear difference equation
of the first order
\begin{equation*}
\Gamma_{j}(z+1)=g_{j}(z)\Gamma_{j}(z),
\end{equation*}
$g_{j}(z)$ being a rational function of $z$ ($j=1,2,\ldots,n$).
\end{definition}

However, we make the following restrictive definition.

\begin{definition}
\label{def:subst_same}
Two systems
$f(z+1)=B(z)f(z)$ and $h(z+1)=C(z)h(z)$
of essentially the same
are said to be \emph{substantially the same} 
if we can take
\begin{equation*}
\Gamma_{1}(z)=
\Gamma_{2}(z)=\cdots=
\Gamma_{n}(z)
\end{equation*}
in the transformation \eqref{trans:inDef}.
\end{definition}

We shall show in Theorem \ref{thm:main} that 
it becomes a condition for determining 
values of accessory parameters
that the system of difference equations
for the series solution near $x=1$
and that 
for the series solution near $x=\infty$
are substantially the same.
Moreover, we shall show in Theorem \ref{thm:realize}
that the system of the Okubo normal form determined
in Theorem \ref{thm:main} coincides with the system \eqref{sys:zero}
up to a diagonal transformation.

\section{Preliminary}
In this section we give a parametrization 
of the coefficient matrix $A$ in \eqref{sys:gene}.

\begin{proposition} \label{prop:expr_a_kai}
Assume the condition \eqref{cond:base}.
For $\xi=\pm c, \pm d$ let $v_{\xi}$ be a left eigenvector of 
the matrix $A$  
in \eqref{sys:gene} 
with respect to the eigenvalue $\xi$.
Provided that $v_{\pm c}
  =\begin{pmatrix} 
     v^{1}_{\pm c} & v^{2}_{\pm c} & v^{3}_{\pm c} & v^{4}_{\pm c}
   \end{pmatrix}$ satisfy
\begin{equation}
v^{k}_{\pm c}\ne 0 \ \ (k=1,2,3,4)
\quad \mbox{and} \quad
\det \begin{pmatrix} 
     v^{l}_{-c} & v^{l+1}_{-c} \\[\jot]
     v^{l}_{c} & v^{l+1}_{c}
   \end{pmatrix}\ne 0 \ \ (l=1,3),
\label{cond:v_c}
\end{equation}
then
there exists a diagonal matrix $D$ such that
$DAD^{-1}$ has a parametrization of the form
\begin{equation}
DAD^{-1}
=
\begin{pmatrix}
a J &
   \begin{matrix}
     \frac{(b-c)r_{2}-(b+c)r_{3}}{r_{1}-r_{2}}  
   & \frac{(b+c)r_{2}-(b-c)r_{4}}{r_{2}-r_{1}}  
   \\[\jot]
     \frac{(b-c)r_{1}-(b+c)r_{3}}{r_{2}-r_{1}}  
   & \frac{(b+c)r_{1}-(b-c)r_{4}}{r_{1}-r_{2}}  
   \end{matrix}
\\
\begin{matrix}
    \frac{(a+c)r_{1}-(a-c)r_{4}}{r_{4}-r_{3}}  
  & \frac{(a-c)r_{2}-(a+c)r_{4}}{r_{3}-r_{4}}  
  \\[\jot]
    \frac{(a+c)r_{1}-(a-c)r_{3}}{r_{3}-r_{4}}  
  & \frac{(a-c)r_{2}-(a+c)r_{3}}{r_{4}-r_{3}}  
\end{matrix}
 & b J
\end{pmatrix},
\label{parametrization_r}
\end{equation}
where
the $r_{k}$'s
satisfy
\begin{equation}
r_{1}r_{2}r_{3}r_{4}\ne0, \quad
r_{1}-r_{2}\ne0, \quad 
r_{3}-r_{4}\ne0
\label{cond:r}
\end{equation}
and
\begin{equation}
\frac{
\begin{aligned}
&
 (a+b+c)^2r_{1}r_{3}
-(a-b+c)^2r_{1}r_{4}
\\[-\jot]
&
\qquad{}
-(a-b-c)^2r_{2}r_{3}
+(a+b-c)^2r_{2}r_{4}
-4abr_{1}r_{2}
-4abr_{3}r_{4}
\end{aligned}
}{(r_{1}-r_{2})(r_{3}-r_{4})}
=d^2.
\label{cond:d}
\end{equation}
\end{proposition}

\begin{proof}
Put 
$L=\begin{pmatrix} 
     v_{-c} \\ v_{c} \\ v_{-d} \\ v_{d}
   \end{pmatrix}$
and 
$D=\begin{pmatrix} 
     v^{1}_{c} & & & \\
     & v^{2}_{c} & & \\
     & & v^{3}_{c} & \\
     & & & v^{4}_{c}
   \end{pmatrix}$,
and write
\begin{equation*}
L
=
\begin{pmatrix}
  L_{11} & L_{12}   \\
  L_{21} & L_{22}   \\
\end{pmatrix}
,
\quad
L^{-1}
=
\begin{pmatrix}
  L^{\prime}_{11} & L^{\prime}_{12}   \\
  L^{\prime}_{21} & L^{\prime}_{22}   \\
\end{pmatrix},
\quad
D
=
\begin{pmatrix}
  D_{1} &    \\
   & D_{2}   \\
\end{pmatrix},
\end{equation*}
where all the submatrices are $2\times2$ matrices.
Note that $D_{1}$, $D_{2}$, $L_{11}$, $L_{12}$ are invertible 
because of \eqref{cond:v_c}.
Set
\begin{equation*}
\tilde{A}
=
\begin{pmatrix}
  D &   \\
    & L \\
\end{pmatrix}
\begin{pmatrix}
  A & I_{4}  \\
  d^2I_{4}-A^2 & -A \\
\end{pmatrix}
\begin{pmatrix}
  D &   \\
    & L \\
\end{pmatrix}^{-1}.
\end{equation*}
Then $\tilde{A}$ 
satisfies 
\begin{equation*}
\tilde{A}^2=d^2I_{8}.  
\end{equation*}

Since $\tilde{A}$ has an expression of the form 
\begin{equation*}
\tilde{A}
=
\begin{pmatrix}
  aJ & D_{1}A_{12}{D_{2}}^{-1} & D_{1}L^{\prime}_{11} & D_{1}L^{\prime}_{12}
  \\
  D_{2}A_{21}{D_{1}}^{-1} & bJ & D_{2}L^{\prime}_{21} & D_{2}L^{\prime}_{22}
  \\
  \bigl(d^2-c^2\bigr) L_{11}{D_{1}}^{-1} &
  \bigl(d^2-c^2\bigr) L_{12}{D_{2}}^{-1} & cJ & \\
  & & & dJ \\
\end{pmatrix},
\end{equation*}
we obtain
\begin{subequations}
\begin{align}
L_{11}{D_{1}}^{-1}D_{1}A_{12}{D_{2}}^{-1} 
+ bL_{12}{D_{2}}^{-1}J + cJL_{12}{D_{2}}^{-1}&=O,
\label{rel:32block}
\\
L_{12}{D_{2}}^{-1}D_{2}A_{21}{D_{1}}^{-1}
+ aL_{11}{D_{1}}^{-1}J + cJL_{11}{D_{1}}^{-1}&=O
\label{rel:31block}
\end{align}
\end{subequations}
from the $(3,2)$-block,
the $(3,1)$-block of $\tilde{A}^2=d^2I_{8}$.
Setting
\begin{equation*}
r_{k}=\frac{v^{k}_{-c}}{v^{k}_{c}} \quad (k=1,2,3,4),
\end{equation*}
we have
\begin{equation}
L_{11}{D_{1}}^{-1}
=\begin{pmatrix}
  r_{1} & r_{2} \\
  1     & 1     \\
 \end{pmatrix}, 
\quad
L_{12}{D_{2}}^{-1}
=\begin{pmatrix}
  r_{3} & r_{4} \\
  1     & 1     \\
 \end{pmatrix}.
\label{expr:by_r}
\end{equation}
Note that $r_{1}r_{2}r_{3}r_{4}\ne0$ and
\begin{equation*}
r_{1}-r_{2}
=
\det\bigl(L_{11}{D_{1}}^{-1}\bigr)\ne0,
\quad
r_{3}-r_{4}
=
\det\bigl(L_{12}{D_{2}}^{-1}\bigr)\ne0
\end{equation*}
by the condition \eqref{cond:r}.
Substituting \eqref{expr:by_r} into \eqref{rel:32block}--\eqref{rel:31block}, 
we obtain
\begin{align*}
D_{1}A_{12}{D_{2}}^{-1} 
&=
-\begin{pmatrix}
  r_{1} & r_{2} \\
  1     & 1     \\
 \end{pmatrix}^{-1} 
\left( b\begin{pmatrix}
  r_{3} & -r_{4} \\
  1     & -1     \\
 \end{pmatrix}
  + c\begin{pmatrix}
  r_{3} & r_{4} \\
  -1     & -1     \\
 \end{pmatrix} \right),
\\
D_{2}A_{21}{D_{1}}^{-1}
&=
-\begin{pmatrix}
  r_{3} & r_{4} \\
  1     & 1     \\
 \end{pmatrix}^{-1}
\left( a\begin{pmatrix}
  r_{1} & -r_{2} \\
  1     & -1     \\
 \end{pmatrix} 
 + c\begin{pmatrix}
  r_{1} & r_{2} \\
  -1     & -1     \\
 \end{pmatrix} \right),
\end{align*}
which give \eqref{parametrization_r}.

Moreover, 
as the $(1,1)$-block,
the $(1,2)$-block of $\tilde{A}^2=d^2I_{8}$
we have
\begin{align*}
a^2I_{2}
+ D_{1}A_{12}{D_{2}}^{-1} D_{2}A_{21}{D_{1}}^{-1}
+ \bigl(d^2-c^2\bigr)D_{1}L^{\prime}_{11} L_{11}{D_{1}}^{-1}&=d^2I_{2},
\\
aJD_{1}A_{12}{D_{2}}^{-1}
+ bD_{1}A_{12}{D_{2}}^{-1}J
+ \bigl(d^2-c^2\bigr)D_{1}L^{\prime}_{11} L_{12}{D_{2}}^{-1}&=O.
\end{align*}
Combining these, we obtain
\begin{align*}
&
a^2I_{2}
+ D_{1}A_{12}{D_{2}}^{-1} D_{2}A_{21}{D_{1}}^{-1}
\\[-\jot]
&\quad{}
- \bigl( aJD_{1}A_{12}{D_{2}}^{-1} + bD_{1}A_{12}{D_{2}}^{-1}J 
         \bigr)D_{2}{L_{12}}^{-1} L_{11}{D_{1}}^{-1}=d^2I_{2}.
\end{align*}
A diagonal element of
this relation gives \eqref{cond:d}.
This completes the proof.
\end{proof}

We put
\begin{equation}
A_{1}=
\begin{pmatrix}
a J &
   \begin{matrix}
     \frac{(b-c)r_{2}-(b+c)r_{3}}{r_{1}-r_{2}}  
   & \frac{(b+c)r_{2}-(b-c)r_{4}}{r_{2}-r_{1}}  
   \\[\jot]
     \frac{(b-c)r_{1}-(b+c)r_{3}}{r_{2}-r_{1}}  
   & \frac{(b+c)r_{1}-(b-c)r_{4}}{r_{1}-r_{2}}  
   \end{matrix}
\\
\begin{matrix}
    \frac{(a+c)r_{1}-(a-c)r_{4}}{r_{4}-r_{3}}  
  & \frac{(a-c)r_{2}-(a+c)r_{4}}{r_{3}-r_{4}}  
  \\[\jot]
    \frac{(a+c)r_{1}-(a-c)r_{3}}{r_{3}-r_{4}}  
  & \frac{(a-c)r_{2}-(a+c)r_{3}}{r_{4}-r_{3}}  
\end{matrix}
 & b J
\end{pmatrix},
\label{coef:A1}
\end{equation}
where we regard 
$r_{k}$ $(k=1,2,3,4)$
as arbitrary parameters
not relating to
$v^{k}_{\pm c}$ $(k=1,2,3,4)$
but satisfying \eqref{cond:r} and \eqref{cond:d}.
In what follows,
we investigate the system
\begin{equation}
(xI_{4}-T)\frac{dy}{dx}
=A_{1}y,
\label{sys:one}
\end{equation}
where $T=
\begin{pmatrix}
  0I_{2} &       \\
    & 1I_{2}     \\
   \end{pmatrix}$.

\begin{remark}
\label{rem:no_AP}
All the elements of the non-diagonal blocks of $A_{1}$
are expressed by
a ratio of homogeneous polynomials in $r_{k}$ $(k=1,2,3,4)$ of degree one.
Taking the relation \eqref{cond:d} into account,
we find that they can be expressed by two parameters,
for instance, by
\begin{equation*}
t_{k}=\frac{r_{k}}{r_{4}} \quad (k=1,2)
\end{equation*}
in addition to the local exponents $a$, $b$, $c$, $d$. 
The two parameters are accessory parameters.
\end{remark}

\begin{remark}
The local exponents $(0,0,a,-a)$, $(0,0,b,-b)$, $(c,-c,d,-d)$
of \eqref{sys:gene} are essential for the parametrization
\eqref{parametrization_r}.
A generalization of the local exponents to
\begin{equation*}
(0,0,a,a^{\prime}), \quad 
(0,0,b,b^{\prime}), \quad 
(c,c^{\prime},d,d^{\prime})
\end{equation*}
is impossible.
\end{remark}

\section{Local solutions}
In this section we study
local solutions of \eqref{sys:one}
near its singular points.
The system \eqref{sys:one} has regular singularity
at $x=0$, $1$ and $\infty$.
Assume the condition \eqref{cond:base}.
Near $x=1$ it has solutions of the form
\begin{equation*}
y(x)=\sum_{r=0}^{\infty}g(r)(x-1)^{r+\rho}
, \quad
\rho=0,\,\pm b,
\end{equation*}
where the coefficient vectors $g(r)$ ($r=0,1,2,\ldots$)
are determined by the system of difference equations
\begin{equation}
\left\{
  \begin{aligned}
   (r+\rho+1)(T-I_{4})g(r+1)&=\{(r+\rho)I_{4}-A_{1}\}g(r),
   \\
   \rho(T-I_{4})g(0)&=0.
  \end{aligned}
\right. \label{sabun:g}
\end{equation}
Near $x=\infty$ the system \eqref{sys:one} has solutions of the form  
\begin{equation*}
y(x)=\sum_{s=0}^{\infty}h(s)(x-1)^{-s-\sigma}
, \quad
\sigma=\pm c,\,\pm d,
\end{equation*}
where the coefficient vectors $h(s)$ ($s=0,1,2,\ldots$)
are determined by the system of difference equations
\begin{equation}
\left\{
  \begin{aligned}
   \{(s+\sigma+1)I_{4}+A_{1}\}h(s+1)&=(s+\sigma)(T-I_{4})h(s),
   \\
   (\sigma I_{4}+A_{1})h(0)&=0.
  \end{aligned}
\right. \label{sabun:h}
\end{equation}
Note that we express the local solutions near $x=\infty$
by means of powers
not in $x^{-1}$ but in $(x-1)^{-1}$.

In the system \eqref{sabun:g} we set
$z=r+\rho$ and $\tilde{g}(z)=g(z-\rho)$.
Then we obtain the system of difference equations
\begin{equation}
(z+1)(T-I_{4})\tilde{g}(z+1)=(zI_{4}-A_{1})\tilde{g}(z).
\label{sabun:g_z}
\end{equation}
Besides,
in the system \eqref{sabun:h} we set
$z=s+\sigma+1$ and $\tilde{h}(z)=h(z-1-\sigma)$.
Then we obtain the system of difference equations
\begin{equation}
(zI_{4}+A_{1})\tilde{h}(z+1)=(z-1)(T-I_{4})\tilde{h}(z).
\label{sabun:h_z}
\end{equation}
Since
$T-I_{4}=\begin{pmatrix} -1I_{2} &  \\ & 0I_{2} \end{pmatrix}$,
both
\eqref{sabun:g_z} and \eqref{sabun:h_z} are reducible
to a size two system of difference equations.
Indeed, we can write 
\eqref{sabun:g_z}, \eqref{sabun:h_z} in the form
\begin{gather}
(z+1)(zI_{4}-A_{1})^{-1}(T-I_{4})\tilde{g}(z+1)=\tilde{g}(z),
\label{sabun:gg_z}
\\
\tilde{h}(z+1)=(z-1)(zI_{4}+A_{1})^{-1}(T-I_{4})\tilde{h}(z).
\label{sabun:hh_z}
\end{gather}
Writing
\begin{equation*}
\tilde{g}(z)=\begin{pmatrix} \tilde{g}_{1}(z) \\ \tilde{g}_{2}(z) \end{pmatrix}
, \quad
\tilde{h}(z)=\begin{pmatrix} \tilde{h}_{1}(z) \\ \tilde{h}_{2}(z) \end{pmatrix}
\end{equation*}
and
\begin{equation}
(zI_{4}-A_{1})^{-1}
=\frac{1}{(z^2-c^2)(z^2-d^2)}
\begin{pmatrix} A_{11}(z) & A_{12}(z) \\
                A_{21}(z) & A_{22}(z) \end{pmatrix},
\label{inv:z-A1}
\end{equation}
where $\tilde{g}_{j}(z)$ and $\tilde{h}_{j}(z)$ ($j=1,2$)
are $2$-dimensional, and
$A_{jk}(z)$ ($j,k=1,2$) are $2\times2$ matrices,
we obtain
\begin{subequations}
\begin{align}
-\frac{z+1}{(z^2-c^2)(z^2-d^2)}
A_{11}(z)\tilde{g}_{1}(z+1)
&=\tilde{g}_{1}(z),
\label{sabun:gg1}
   \\
-\frac{z+1}{(z^2-c^2)(z^2-d^2)}
A_{21}(z)\tilde{g}_{1}(z+1)
&=\tilde{g}_{2}(z)
\label{sabun:gg2}
\end{align}
\end{subequations}
from \eqref{sabun:gg_z},
and
\begin{subequations}
\begin{align}
\tilde{h}_{1}(z+1)
&=\frac{z-1}{(z^2-c^2)(z^2-d^2)}
A_{11}(-z)\tilde{h}_{1}(z),
\label{sabun:hh1}
   \\
\tilde{h}_{2}(z+1)
&=\frac{z-1}{(z^2-c^2)(z^2-d^2)}
A_{21}(-z)\tilde{h}_{1}(z)
\label{sabun:hh2}
\end{align}
\end{subequations}
from \eqref{sabun:hh_z}.

\section{Main results}
Recall Definition \ref{def:subst_same}.

\begin{theorem} \label{thm:main}
Suppose that the system \eqref{sys:one} satisfies
\eqref{cond:base},
\eqref{cond:r} and \eqref{cond:d}.
The systems \eqref{sabun:gg1} and \eqref{sabun:hh1}
for \eqref{sys:one} 
are substantially the same if and only if
the coefficient matrix $A_{1}$
in \eqref{sys:one} 
is equal to
\begin{equation}
\begin{pmatrix}
a J &
   \begin{matrix}
      \frac{(a-b+c)^2-d^2}{4a} & \frac{(a+b+c)^2-d^2}{4a}
   \\[\jot]
      \frac{d^2-(a+b-c)^2}{4a} & \frac{d^2-(a-b-c)^2}{4a}
   \end{matrix}
\\
\begin{matrix}
      \frac{(a-b-c)^2-d^2}{4b} & \frac{(a+b+c)^2-d^2}{4b}
  \\[\jot]
      \frac{d^2-(a+b-c)^2}{4b} & \frac{d^2-(a-b+c)^2}{4b}
\end{matrix}
 & b J
\end{pmatrix}
\label{coef:A1_noap}
\end{equation}
that is obtained by substituting
\begin{subequations}
\begin{align}
r_{1}&=
\frac{(a+b-c)^2-d^2}{(a+b+c)^2-d^2}
r_{4},
\quad
r_{2}=
\frac{(a-b+c)^2-d^2}
     {(a-b-c)^2-d^2}
r_{4},
\label{expr:r12by4}
\\
r_{3}&=
\frac{\bigl((a+b-c)^2-d^2\bigr)\bigl((a-b+c)^2-d^2\bigr)}
     {\bigl((a+b+c)^2-d^2\bigr)\bigl((a-b-c)^2-d^2\bigr)}
r_{4}
\label{expr:r3by4}
\end{align}
\end{subequations}
or
\begin{subequations}
\begin{align}
r_{1}&=
\frac{(a-b-c)^2-d^2}{(a-b+c)^2-d^2}
r_{3},
\quad
r_{2}=
\frac{(a+b+c)^2-d^2}{(a+b-c)^2-d^2}
r_{3},
\label{expr:r12by3}
\\
r_{4}&=
\frac{\bigl((a+b+c)^2-d^2\bigr)\bigl((a-b-c)^2-d^2\bigr)}
     {\bigl((a+b-c)^2-d^2\bigr)\bigl((a-b+c)^2-d^2\bigr)}
r_{3}
\label{expr:r4by3}
\end{align}
\end{subequations}
into \eqref{coef:A1},
where the denominators
$\bigl((a+b+c)^2-d^2\bigr)\bigl((a-b-c)^2-d^2\bigr)$ and 
$\bigl((a+b-c)^2-d^2\bigr)\bigl((a-b+c)^2-d^2\bigr)$
never vanish simultaneously.
\end{theorem}

\begin{theorem} \label{thm:realize}
If
\begin{equation}
a=\lambda_{++++}, \quad
b=\lambda_{-+-+}, \quad
c=\lambda_{++--}, \quad
d=\lambda_{+--+}
\label{value:abcd}
\end{equation}
and these values satisfy both of the conditions
\begin{subequations}
\begin{align}
\bigl((a+b+c)^2-d^2\bigr)\bigl((a-b-c)^2-d^2\bigr)&\ne0
\label{cond:abcd_a}
\\
\intertext{and}
\bigl((a+b-c)^2-d^2\bigr)\bigl((a-b+c)^2-d^2\bigr)&\ne0
\label{cond:abcd_b}
\end{align}
\end{subequations}
in addition to \eqref{cond:base},
then
the system \eqref{sys:one} with 
the coefficient matrix \eqref{coef:A1_noap}
coincides with
the system \eqref{sys:zero} up to a diagonal transformation.
\end{theorem}

\section{Proofs}
Recall
$A_{1}
=
\begin{pmatrix}
  aJ     & A^{\prime}_{12} \\
  A^{\prime}_{21} & bJ     \\
\end{pmatrix}$,
where $J=\begin{pmatrix} 1 & \\ & -1 \end{pmatrix}$ and
\begin{align*}
A^{\prime}_{12}
&=
\begin{pmatrix}
  \frac{(b-c)r_{2}-(b+c)r_{3}}{r_{1}-r_{2}}  
& \frac{(b+c)r_{2}-(b-c)r_{4}}{r_{2}-r_{1}}  
\\[\jot]
  \frac{(b-c)r_{1}-(b+c)r_{3}}{r_{2}-r_{1}}  
& \frac{(b+c)r_{1}-(b-c)r_{4}}{r_{1}-r_{2}}  
\end{pmatrix},
\\
A^{\prime}_{21}
&=
\begin{pmatrix}
  \frac{(a+c)r_{1}-(a-c)r_{4}}{r_{4}-r_{3}}  
& \frac{(a-c)r_{2}-(a+c)r_{4}}{r_{3}-r_{4}}  
\\[\jot]
  \frac{(a+c)r_{1}-(a-c)r_{3}}{r_{3}-r_{4}}  
& \frac{(a-c)r_{2}-(a+c)r_{3}}{r_{4}-r_{3}}  
\end{pmatrix}.
\end{align*}
Set
\begin{align*}
\varepsilon&=b(a+c)r_{1}+b(a-c)r_{2}-a(b+c)r_{3}-a(b-c)r_{4},
\\
\delta&=r_{1}r_{2}-r_{3}r_{4},
\end{align*}
and
\begin{equation*}
\varepsilon^{\prime}
=\frac{2\varepsilon}{(r_{1}-r_{2})(r_{3}-r_{4})}
, \quad
\delta^{\prime}
=\frac{2\delta}{(r_{1}-r_{2})(r_{3}-r_{4})}.
\end{equation*}

\begin{lemma}
\label{lem:A11}
Assume the conditions
\eqref{cond:base},
\eqref{cond:r} and \eqref{cond:d}.
When we write the submatrix $A_{11}(z)$ in \eqref{inv:z-A1}
in the form
\begin{equation*}
A_{11}(z)=z^3I_{2}+z^2Q_{11}+zR_{11}+S_{11},
\end{equation*}
we have
\begin{subequations}
\begin{align}
Q_{11}
&=aJ,
\label{expr:A11Q}
\\
R_{11}
&=\frac{1}{2}\bigl(a^2-b^2-c^2-d^2\bigr)I_{2}
+\varepsilon^{\prime}
 \begin{pmatrix}
   & r_{2} \\ r_{1} &  
 \end{pmatrix}
-b\delta^{\prime}
 \begin{pmatrix}
   c & a+c \\ a-c & -c 
 \end{pmatrix},
\label{expr:A11R}
\\
S_{11}
&=
bA^{\prime}_{12}JA^{\prime}_{21}
-ab^2J-2abc\delta^{\prime} I_{2}.
\label{expr:A11S}
\end{align}
\end{subequations}
Moreover, we have
\begin{equation*}
{A_{11}(z)}^{-1}
=\frac{1}{(z^2-b^2)(z^2-c^2)(z^2-d^2)}
\bigl(z^3I_{2}+z^2\tilde{Q}_{11}+z\tilde{R}_{11}+\tilde{S}_{11}\bigr),
\end{equation*}
where
\begin{subequations}
\begin{align}
\tilde{Q}_{11}
&=-aJ,
\label{expr:invA11Q}
\\
\tilde{R}_{11}
&=\frac{1}{2}\bigl(a^2-b^2-c^2-d^2\bigr)I_{2}
-\varepsilon^{\prime}
 \begin{pmatrix}
   & r_{2} \\ r_{1} & 
 \end{pmatrix}
+b\delta^{\prime}
 \begin{pmatrix}
   c & a+c \\ a-c & -c
 \end{pmatrix},
\label{expr:invA11R}
\\
\tilde{S}_{11}
&=
ab^2J-bA^{\prime}_{12}JA^{\prime}_{21}.
\label{expr:invA11S}
\end{align}
\end{subequations}
\end{lemma}

\begin{proof}
Since $\bigl({A_{1}}^2-c^2I_{4}\bigr)\bigl({A_{1}}^2-d^2I_{4}\bigr)=O$,
we have
\begin{equation*}
\bigl(zI_{4}-{A_{1}}^2\bigr)^{-1}
=\frac{1}{(z-c^2)(z-d^2)}\bigl\{\bigl(z-c^2-d^2\bigr)I_{4}+{A_{1}}^2\bigr\},
\end{equation*}
and hence
\begin{align*}
&
(zI_{4}-A_{1})^{-1}
=
(zI_{4}+A_{1})
\bigl(z^2I_{4}-{A_{1}}^2\bigr)^{-1}
\\
&=
\frac{1}{(z^2-c^2)(z^2-d^2)}(zI_{4}+A_{1})
\bigl\{\bigl(z^2-c^2-d^2\bigr)I_{4}+{A_{1}}^2\bigr\}
\\
&=
\frac{1}{(z^2-c^2)(z^2-d^2)}
  \bigl\{z^3I_{4}+z^2A_{1}+z\bigl({A_{1}}^2-\bigl(c^2+d^2\bigr)I_{4}\bigr)
  +{A_{1}}^3-\bigl(c^2+d^2\bigr)A_{1}\bigr\}.
\end{align*}
From the $(1,1)$-block of this formula we have
\begin{align*}
Q_{11}
&=aJ,
\\
R_{11}
&=
A^{\prime}_{12}A^{\prime}_{21}+\bigl(a^2-c^2-d^2\bigr)I_{2},
\\
S_{11}
&=
bA^{\prime}_{12}JA^{\prime}_{21}
+aJA^{\prime}_{12}A^{\prime}_{21}+aA^{\prime}_{12}A^{\prime}_{21}J
+a\bigl(a^2-c^2-d^2\bigr)J.
\end{align*}
Besides, with the aid of a computer algebra system,
we can easily check that
\begin{equation*}
A^{\prime}_{12}A^{\prime}_{21}
=\frac{1}{2}\bigl(c^2+d^2-a^2-b^2\bigr)I_{2}
+\varepsilon^{\prime}
 \begin{pmatrix}
   & r_{2} \\ r_{1} &
 \end{pmatrix}
-b\delta^{\prime}
 \begin{pmatrix}
   c & a+c \\ a-c & -c
 \end{pmatrix}
\end{equation*}
holds,
where we have used \eqref{cond:d} for $d^2$.
Combining these expressions,
we can obtain \eqref{expr:A11Q}--\eqref{expr:A11S}.

As for ${A_{11}(z)}^{-1}$, we have
\begin{equation*}
\begin{gathered}
{A_{11}(z)}^{-1}
=
\frac{1}{\det A_{11}(z)}\adj A_{11}(z),
\\
\adj A_{11}(z)
=
z^3I_{2}+z^2((\tr Q_{11})I_{2}-Q_{11})
+z((\tr R_{11})I_{2}-R_{11})+(\tr S_{11})I_{2}-S_{11}.
\end{gathered}
\end{equation*}
Here
the determinant of $A_{11}(z)$ follows from the equality
\begin{equation*}
\begin{pmatrix} A_{11}(z) & A_{12}(z) \\ & I_{2} \end{pmatrix}(zI_{4}-A_{1})
=
\begin{pmatrix} \bigl(z^2-c^2\bigr)\bigl(z^2-d^2\bigr)I_{2} & \\
                -A^{\prime}_{21} & zI_{2}-bJ \end{pmatrix}.
\end{equation*}
Taking the determinant of both sides, we have
\begin{equation*}
\det A_{11}(z)\cdot \det(zI_{4}-A_{1})
=\bigl(z^2-c^2\bigr)^2\bigl(z^2-d^2\bigr)^2\cdot \det(zI_{2}-bJ)
\end{equation*}
and hence
\begin{equation*}
\det A_{11}(z)=\bigl(z^2-b^2\bigr)\bigl(z^2-c^2\bigr)\bigl(z^2-d^2\bigr).
\end{equation*}
From \eqref{expr:A11Q}--\eqref{expr:A11S}
it is trivial that
\begin{equation*}
\tr Q_{11}=0, \quad
\tr R_{11}=a^2-b^2-c^2-d^2, \quad
\tr S_{11}=b\cdot \tr\bigl(A^{\prime}_{12}JA^{\prime}_{21}\bigr)
           -4abc\delta^{\prime}
\end{equation*}
hold.
By direct calculation
we can obtain $\tr\bigl(A^{\prime}_{12}JA^{\prime}_{21}\bigr)
=2ac\delta^{\prime}$ and hence
\begin{equation*}
\tr S_{11}=-2abc\delta^{\prime}.
\end{equation*}
Substituting these,
we obtain \eqref{expr:invA11Q}--\eqref{expr:invA11S}.
This completes the proof.
\end{proof}

\begin{lemma} \label{lem:cond_same}
Assume the conditions
\eqref{cond:base},
\eqref{cond:r} and \eqref{cond:d}.
The systems \eqref{sabun:gg1} and \eqref{sabun:hh1}
are 
substantially the same 
if and only if the condition
\begin{equation}
\varepsilon=0
\quad \mbox{and} \quad 
\delta=0
\label{cond:lm}
\end{equation}
holds.
\end{lemma}

\begin{proof}
By Lemma \ref{lem:A11}
we can write
the systems \eqref{sabun:gg1} and \eqref{sabun:hh1}
in the form
\begin{align*}
\tilde{g}_{1}(z+1)
&=
-\frac{1}{(z+1)(z^2-b^2)}
\bigl(z^3I_{2}-z^2aJ+z\tilde{R}_{11}+\tilde{S}_{11}\bigr)
\tilde{g}_{1}(z)
, \\
\tilde{h}_{1}(z+1)
&=
-\frac{z-1}{(z^2-c^2)(z^2-d^2)}
\bigl(z^3I_{2}-z^2aJ+zR_{11}-S_{11}\bigr)
\tilde{h}_{1}(z).
\end{align*}
From these expressions it is obvious
that the systems \eqref{sabun:gg1} and \eqref{sabun:hh1}
are substantially the same if \eqref{cond:lm} holds. 
So, we only have to show that the condition \eqref{cond:lm} holds
if \eqref{sabun:gg1} and \eqref{sabun:hh1} are substantially the same.
Throughout the proof,
we write
\begin{align*}
z^3I_{2}-z^2aJ+z\tilde{R}_{11}+\tilde{S}_{11}
&=
\begin{pmatrix}
  b_{11}(z) & b_{12}(z) \\
  b_{21}(z) & b_{22}(z) \\
\end{pmatrix},
\\
z^3I_{2}-z^2aJ+zR_{11}-S_{11}
&=
\begin{pmatrix}
  c_{11}(z) & c_{12}(z) \\
  c_{21}(z) & c_{22}(z) \\
\end{pmatrix}.
\end{align*}
Assume that
the systems \eqref{sabun:gg1} and \eqref{sabun:hh1}
are substantially the same.
Then we have
\begin{equation*}
c_{jk}(z)b_{lm}(z)-b_{jk}(z)c_{lm}(z)=0
\quad (j,k,l,m=1,2).
\end{equation*}
By direct calculation we obtain
\begin{align*}
c_{11}(z)b_{22}(z)-b_{11}(z)c_{22}(z)
&=-4bc\delta^{\prime}z^4+\cdots,
\\
c_{12}(z)b_{22}(z)-b_{12}(z)c_{22}(z)
&=2 \bigl\{r_{2}\varepsilon^{\prime}-(a+c)b\delta^{\prime}\bigr\}z^4+\cdots,
\\
c_{21}(z)b_{22}(z)-b_{21}(z)c_{22}(z)
&=2 \bigl\{r_{1}\varepsilon^{\prime}-(a-c)b\delta^{\prime}\bigr\}z^4+\cdots,
\end{align*}
which gives $\delta^{\prime}=0$ and $\varepsilon^{\prime}=0$.
This completes the proof.
\end{proof}

\begin{proof}[Proof of Theorem \ref{thm:main}]
It is easy to check that 
by substituting \eqref{expr:r12by4}--\eqref{expr:r3by4}
or \eqref{expr:r12by3}--\eqref{expr:r4by3}
into \eqref{coef:A1}
we obtain \eqref{coef:A1_noap}.
So, by Lemma \ref{lem:cond_same} we only have to show that
either or both of
\eqref{expr:r12by4}--\eqref{expr:r3by4}
and \eqref{expr:r12by3}--\eqref{expr:r4by3}
holds if and only if \eqref{cond:lm} holds.

First, we assume \eqref{cond:lm}. 
Using $\varepsilon$ and $\delta$,
we can write the numerator of the left hand side of
\eqref{cond:d} as
\begin{align*}
&
(r_{1}-r_{2})(r_{3}-r_{4})d^2
\\
&=
 (a+b+c)^2r_{1}r_{3}
-(a-b+c)^2r_{1}r_{4}
-(a-b-c)^2r_{2}r_{3}
+(a+b-c)^2r_{2}r_{4}
\\
&\phantom{{}={}}
{}
-4abr_{1}r_{2}
-4abr_{3}r_{4}
\\
&
=
\bigl\{
 4b(a+c)r_{1}
+(a-b-c)^2r_{3}
-(a+b-c)^2r_{4}
\bigr\}(r_{1}-r_{2})
-4r_{1}\varepsilon
+4ab\delta
\\
&
=
\bigl\{
-4b(a-c)r_{2}
+(a+b+c)^2r_{3}
-(a-b+c)^2r_{4}
\bigr\}(r_{1}-r_{2})
-4r_{2}\varepsilon
+4ab\delta.
\end{align*}
When $\varepsilon=0$, $\delta=0$ and $r_{1}-r_{2}\ne0$,
we have 
\begin{subequations}
\begin{align}
r_{1}
&
=-\frac{
 \bigl((a-b-c)^2-d^2\bigr)r_{3}
-\bigl((a+b-c)^2-d^2\bigr)r_{4}
}{4b(a+c)},
\label{expr:r1}
\\
r_{2}
&
=\frac{
 \bigl((a+b+c)^2-d^2\bigr)r_{3}
-\bigl((a-b+c)^2-d^2\bigr)r_{4}
}{4b(a-c)}.
\label{expr:r2}
\end{align}
\end{subequations}
Substituting these into $\delta=r_{1}r_{2}-r_{3}r_{4}$,
we have
\begin{equation*}
\delta
=
-\frac{
\begin{aligned}
&
\bigl\{
\bigl((a+b+c)^2-d^2\bigr)\bigl((a-b-c)^2-d^2\bigr)r_{3}
\\[-\jot]
&
\qquad\qquad
{}-
\bigl((a+b-c)^2-d^2\bigr)\bigl((a-b+c)^2-d^2\bigr)r_{4}
\bigr\}
(r_{3}-r_{4})
\end{aligned}
}{16b^2(a^2-c^2)}.
\end{equation*}
When $\delta=0$ and $r_{3}-r_{4}\ne0$,
we obtain \eqref{expr:r3by4}
and/or \eqref{expr:r4by3}.
Here the coefficient of $r_{3}$ and that of $r_{4}$
never vanish simultaneously, since
\begin{equation*}
\bigl((a+\iota_{1} b-\iota_{2} c)^2-d^2\bigr)
\bigl((a-\iota_{1} b+\iota_{2} c)^2-d^2\bigr)
=\iota_{1}\iota_{2}16bc
 (a+\iota_{1} b)(a+\iota_{2} c)
\ne0
\end{equation*}
if $(a+\iota_{1} b+\iota_{2} c)^2-d^2=0$,
where
$\iota_{j}=+$ or $-$ $(j=1,2)$. 
Substituting \eqref{expr:r3by4} or \eqref{expr:r4by3} 
into \eqref{expr:r1}--\eqref{expr:r2}, 
we obtain \eqref{expr:r12by4} or \eqref{expr:r12by3}, 
respectively.

Inversely, we assume
\eqref{expr:r12by4}--\eqref{expr:r3by4}
or \eqref{expr:r12by3}--\eqref{expr:r4by3}.
By direct calculation we can easily check 
that $\varepsilon=0$ and $\delta=0$ hold.
This completes the proof.
\end{proof}

\begin{proof}[Proof of Theorem \ref{thm:realize}]
First, substituting 
\eqref{value:abcd} into
\eqref{cond:abcd_a}, \eqref{cond:abcd_b},
we have
\begin{align*}
\bigl((a+b+c)^2-d^2\bigr)\bigl((a-b-c)^2-d^2\bigr)
&=
64\lambda_{0+00}\lambda_{00+0}\lambda_{+-++}\lambda_{++-+}\ne0,
\\
\bigl((a+b-c)^2-d^2\bigr)\bigl((a-b+c)^2-d^2\bigr)
&=
64\lambda_{+000}\lambda_{000+}\lambda_{-+++}\lambda_{+++-}\ne0.
\end{align*}
Substituting 
\eqref{value:abcd} into
\eqref{coef:A1_noap},
we have
\begin{equation*}
A_{1}
=
\begin{pmatrix}
\lambda_{++++} J &
   \begin{matrix}
      \frac{2\lambda_{+000}\lambda_{+++-}}{\lambda_{++++}}
    & \frac{2\lambda_{0+00}\lambda_{++-+}}{\lambda_{++++}}
   \\[\jot]
      \frac{2\lambda_{000+}\lambda_{-+++}}{\lambda_{----}}
    & \frac{2\lambda_{00+0}\lambda_{+-++}}{\lambda_{----}}
   \end{matrix}
\\
\begin{matrix}
      \frac{2\lambda_{00+0}\lambda_{+-++}}{\lambda_{-+-+}}
    & \frac{2\lambda_{0+00}\lambda_{++-+}}{\lambda_{-+-+}}
  \\[\jot]
      \frac{2\lambda_{000+}\lambda_{-+++}}{\lambda_{+-+-}}
    & \frac{2\lambda_{+000}\lambda_{+++-}}{\lambda_{+-+-}}
\end{matrix}
 & \lambda_{-+-+} J
\end{pmatrix}.
\end{equation*}
For this $A_{1}$, taking
\begin{equation*}
D_{1}
=\begin{pmatrix}
   1 & & & \\
   & \frac{\lambda_{-+++}\lambda_{+-++}}{4\lambda_{+000}\lambda_{0+00}} & & \\
   & & \frac{\lambda_{+-++}}{2\lambda_{+000}} & \\
   & & & \frac{\lambda_{-+++}}{2\lambda_{0+00}} \\
 \end{pmatrix},
\end{equation*}
we see that ${D_{1}}^{-1}A_{1}D_{1}$ agrees with
the coefficient matrix $A_{0}$ of \eqref{sys:zero}.
Namely, the transformation $y=D_{1}u$ changes
\eqref{sys:one} to \eqref{sys:zero}.
This completes the proof.
\end{proof}

\section{Some remarks}

\subsection{Local solutions near another singular point}
We can obtain the same result as 
Theorem \ref{thm:main}
from local solutions
\begin{alignat*}{3}
y(x)&=\sum_{r=0}^{\infty}g(r)x^{r+\rho}, & \quad
\rho&=0,\,\pm a, & \quad
\mbox{near } x&=0
\\
\intertext{and}
y(x)&=\sum_{s=0}^{\infty}h(s)x^{-s-\sigma}, & \quad
\sigma&=\pm c,\,\pm d, & \quad
\mbox{near } x&=\infty.
\end{alignat*}
For these solutions,
similarly to \eqref{sabun:g_z} and \eqref{sabun:h_z},
we have
\begin{alignat*}{2}
(z+1)T\hat{g}(z+1)&=(zI_{4}-A_{1})\hat{g}(z) & \quad
\mbox{for } \hat{g}(z)&=g(z-\rho)
\\
\intertext{and}
(zI_{4}+A_{1})\hat{h}(z+1)&=(z-1)T\hat{h}(z) & \quad
\mbox{for } \hat{h}(z)&=h(z-1-\sigma).
\end{alignat*}
Moreover, 
since $T=
\begin{pmatrix}
  0I_{2} &       \\
    & 1I_{2}     \\
   \end{pmatrix}$,
writing
$\hat{g}(z)
=\begin{pmatrix} \hat{g}_{1}(z) \\ \hat{g}_{2}(z) \end{pmatrix}$
and
$\hat{h}(z)
=\begin{pmatrix} \hat{h}_{1}(z) \\ \hat{h}_{2}(z) \end{pmatrix}$,
we have
\begin{subequations}
\begin{align}
\frac{z+1}{(z^2-c^2)(z^2-d^2)}
A_{12}(z)\hat{g}_{2}(z+1)
&=\hat{g}_{1}(z),
\label{sabun:gg3}
   \\ 
\frac{z+1}{(z^2-c^2)(z^2-d^2)}
A_{22}(z)\hat{g}_{2}(z+1)
&=\hat{g}_{2}(z)
\label{sabun:gg4}
\end{align}
\end{subequations}
and
\begin{subequations}
\begin{align}
\hat{h}_{1}(z+1)
=-\frac{z-1}{(z^2-c^2)(z^2-d^2)}
A_{12}(-z)\hat{h}_{2}(z),
\label{sabun:hh3}
   \\ 
\hat{h}_{2}(z+1)
=-\frac{z-1}{(z^2-c^2)(z^2-d^2)}
A_{22}(-z)\hat{h}_{2}(z),
\label{sabun:hh4}
\end{align}
\end{subequations}
where $A_{12}(z)$ and $A_{22}(z)$ are the submatrices in \eqref{inv:z-A1}.
In the same way as the proof of Theorem \ref{thm:main}
we can prove the following theorem.

\begin{theorem}
Suppose that the system \eqref{sys:one} satisfies
\eqref{cond:base},
\eqref{cond:r} and \eqref{cond:d}.
The systems \eqref{sabun:gg4} and \eqref{sabun:hh4}
for \eqref{sys:one} 
are substantially the same if and only if
the coefficient matrix $A_{1}$
in \eqref{sys:one} 
is equal to
\eqref{coef:A1_noap}.
\end{theorem}

\subsection{Relation with the equation treated by Ebisu et~al.}
The present work is inspired by Ebisu et~al.~\cite{EHKOSY19}.
They investigated the scalar differential equation of the fourth order
satisfied by
\begin{equation}
y=
x^{-A_{0}}
{}_{2}F_{1}(A_{-+-+},A_{-++-},1-A_{0};x)
{}_{2}F_{1}(A_{----},A_{--++},1-A_{0};x)
\label{prod:EHKOSY}
\end{equation}
under the parametrization
\begin{equation*}
A_{\varepsilon_{0}\varepsilon_{1}\varepsilon_{2}\varepsilon_{3}}
=\frac{1}{2}\{
      (\varepsilon_{0}A_{0})
     +(\varepsilon_{1}A_{1})
     +(\varepsilon_{2}A_{2})
     +(\varepsilon_{3}A_{3})+1\}, \quad \varepsilon_{j}=\pm
\end{equation*}
in their notation.
Note that the product \eqref{prod:EHKOSY} is equivalent to
\begin{equation}
y=
x^{\gamma^{\prime}-1}
{}_{2}F_{1}\bigl(\alpha_{1},\beta_{1},\gamma^{\prime};x\bigr)
{}_{2}F_{1}\bigl(\alpha_{2},\beta_{2},\gamma^{\prime};x\bigr)
\label{prod:EHKOSYequiv}
\end{equation}
with the condition
\begin{equation}
\gamma^{\prime}=\frac{1}{2}(\alpha_{1}+\alpha_{2}+\beta_{1}+\beta_{2}),
\label{cond:EHKOSY}
\end{equation}
which is different from the condition \eqref{cond:Okubo}.

We here explain the relation between 
our system and their equation.
For the system \eqref{sys:zero} we change the variable $u$ to $v$ by 
\begin{equation*}
u=R^{-1}v, \quad
R=
\begin{pmatrix}
1 & \frac{\lambda_{-+++}\lambda_{+-++}}{4\lambda_{+000}\lambda_{0+00}}
  & \frac{\lambda_{+-++}}{2\lambda_{+000}} 
  & \frac{\lambda_{-+++}}{2\lambda_{0+00}} 
\\[\jot]
1 & \frac{\lambda_{++-+}\lambda_{+++-}}{4\lambda_{00+0}\lambda_{000+}}
  & \frac{\lambda_{+++-}}{2\lambda_{00+0}} 
  & \frac{\lambda_{++-+}}{2\lambda_{000+}} 
\\[\jot]
1 & \frac{\lambda_{-+++}\lambda_{+++-}}{4\lambda_{+000}\lambda_{000+}}
  & \frac{\lambda_{+++-}}{2\lambda_{+000}} 
  & \frac{\lambda_{-+++}}{2\lambda_{000+}} 
\\[\jot]
1 & \frac{\lambda_{+-++}\lambda_{++-+}}{4\lambda_{0+00}\lambda_{00+0}} 
  & \frac{\lambda_{+-++}}{2\lambda_{00+0}} 
  & \frac{\lambda_{++-+}}{2\lambda_{0+00}}
\end{pmatrix}
\end{equation*}
under the conditions
$\lambda_{+000}\lambda_{0+00}\lambda_{00+0}\lambda_{000+}\ne0$ and
\begin{equation*}
\det R=\frac{\lambda_{++++}\lambda_{++--}\lambda_{+-+-}\lambda_{+--+}
              {\lambda_{+0-0}}^2{\lambda_{0+0-}}^2}
             {16{\lambda_{+000}}^2{\lambda_{0+00}}^2
                {\lambda_{00+0}}^2{\lambda_{000+}}^2}\ne0.
\end{equation*}
The matrix $R$ consists of left eigenvectors of
the coefficient matrix $A_{0}$ of \eqref{sys:zero} to satisfy
\begin{equation*}
RA_{0}R^{-1}=
\begin{pmatrix}
\lambda_{++--} J &                \\
              & \lambda_{+--+} J  \\
\end{pmatrix}.
\end{equation*}
Then $v$ satisfies the system of differential equations
\begin{equation}
\frac{dv}{dx}
=R(xI_{4}-T)^{-1}R^{-1}
\begin{pmatrix}
\lambda_{++--} J &                \\
              & \lambda_{+--+} J  \\
\end{pmatrix}
v.
\label{sys:v}
\end{equation}
Besides, 
by direct calculation
we have
\begin{equation*}
RP^{-1}=
\begin{pmatrix}
  1 & \frac{1}{\alpha_{1}} & \frac{1}{\alpha_{2}}
    & \frac{1}{\alpha_{1}\alpha_{2}}
\\[\jot]
  1 & \frac{1}{\beta_{1}} & \frac{1}{\beta_{2}}
    & \frac{1}{\beta_{1}\beta_{2}}
\\[\jot]
  1 & \frac{1}{\alpha_{1}} & \frac{1}{\beta_{2}}
    & \frac{1}{\alpha_{1}\beta_{2}}
\\[\jot]
  1 & \frac{1}{\beta_{1}} & \frac{1}{\alpha_{2}}
    & \frac{1}{\beta_{1}\alpha_{2}}
\end{pmatrix}.
\end{equation*}
For $w=\begin{pmatrix}
 f_{1}f_{2} \\ x{f_{1}}'f_{2} \\ x f_{1}{f_{2}}' \\ x^2{f_{1}}'{f_{2}}'
 \end{pmatrix}$,
$f_{j}={}_{2}F_{1}(\alpha_{j},\beta_{j},\gamma_{j};x)$ $(j=1,2)$,
using the relation
\begin{equation*}
f_{j}+\frac{x}{\alpha_{j}}{f_{j}}^{\prime}
={}_{2}F_{1}(\alpha_{j}+1,\beta_{j},\gamma_{j};x)
, \quad
f_{j}+\frac{x}{\beta_{j}}{f_{j}}^{\prime}
={}_{2}F_{1}(\alpha_{j},\beta_{j}+1,\gamma_{j};x),
\end{equation*}
we have
\begin{align}
v
&=x^{\gamma_{1}-1}RP^{-1}w
\notag
\\
&=x^{\gamma_{1}-1}
\begin{pmatrix}
{}_{2}F_{1}(\alpha_{1}+1,\beta_{1},\gamma_{1};x)
{}_{2}F_{1}(\alpha_{2}+1,\beta_{2},\gamma_{2};x) \\
{}_{2}F_{1}(\alpha_{1},\beta_{1}+1,\gamma_{1};x)
{}_{2}F_{1}(\alpha_{2},\beta_{2}+1,\gamma_{2};x) \\
{}_{2}F_{1}(\alpha_{1}+1,\beta_{1},\gamma_{1};x)
{}_{2}F_{1}(\alpha_{2},\beta_{2}+1,\gamma_{2};x) \\
{}_{2}F_{1}(\alpha_{1},\beta_{1}+1,\gamma_{1};x)
{}_{2}F_{1}(\alpha_{2}+1,\beta_{2},\gamma_{2};x)
\end{pmatrix},
\label{sol:v}
\end{align}
where
the parameters 
$\alpha_{j}$, $\beta_{j}$, $\gamma_{j}$ ($j=1,2$)
satisfy the condition \eqref{cond:Okubo},
which is written in the form
\begin{equation*}
\gamma_{1}
=\gamma_{2}
=\frac12(\alpha_{1}+1+\alpha_{2}+1+\beta_{1}+\beta_{2}),
\mbox{~etc.}
\end{equation*}
This means that the scalar differential equation of the fourth order 
satisfied by 
a component of $v$ 
is equivalent to
the equation
satisfied by 
\eqref{prod:EHKOSYequiv} with \eqref{cond:EHKOSY}
by suitable change of parameters.
For example, 
the equation satisfied by the first component of $v$ 
agrees with
the equation satisfied by \eqref{prod:EHKOSYequiv}
with $\alpha_{1}$, $\alpha_{2}$ and $\gamma^{\prime}$
replaced by
$\alpha_{1}+1$, $\alpha_{2}+1$ and $\gamma^{\prime}+1$,
respectively.

\appendix

\section{The Dotsenko-Fateev system}
Ebisu et~al.~\cite{EHKOSY19}
have shown that 
the Dotsenko-Fateev equation
is obtained 
from the equation for \eqref{prod:EHKOSY}
by the middle convolution and the addition.

In this appendix we shall show that
the Euler transformation of a certain order
of the system \eqref{sys:v} 
gives the Dotsenko-Fateev system.
It is a size three system of the form
\begin{equation}
\frac{dz}{dx}
=
\left(
\frac{1}{x}
C_{0}
+
\frac{1}{x-1}
C_{1}
\right)
z,
\label{sys:DF}
\end{equation}
where
\begin{equation*}
C_{0}
=
\begin{pmatrix}
 2a+2c+g & 0 & b \\
 0 & 0 & 0 \\
 0 & 2b+g & a+c \\
\end{pmatrix},
\quad
C_{1}
=
\begin{pmatrix}
 0 & 0 & 0 \\
 0 & 2b+2c+g & a \\
 2a+g & 0 & b+c \\
\end{pmatrix}
\end{equation*}
(see Haraoka~\cite{Har13}).
For later use, we remark that
\begin{subequations}
\begin{alignat}{2}
C_{0}-(a+c)I_{3}
&
=
\begin{pmatrix}
 a+c+g & 0 & b \\
 0 & -a-c & 0 \\
 0 & 2b+g & 0 \\
\end{pmatrix}
&&\sim
\begin{pmatrix}
 a+c+g &   &   \\
   & -a-c &   \\
   &  & 0 \\
\end{pmatrix},
\label{mat:simC0C1a}
\\
C_{1}-(b+c)I_{3}
&
=
\begin{pmatrix}
 -b-c & 0 & 0 \\
 0 & b+c+g & a \\
 2a+g & 0 & 0 \\
\end{pmatrix}
&&\sim
\begin{pmatrix}
 b+c+g &   &   \\
   & -b-c &   \\
   &  & 0 \\
\end{pmatrix}
\label{mat:simC0C1b}
\end{alignat}
and
\begin{align}
&
C_{0}+C_{1}-(a+b+2c)I_{3}
\notag
\\
&
=
\begin{pmatrix}
 a-b+g & 0 & b \\
 0 & b-a+g & a \\
 2a+g & 2b+g & 0 \\
\end{pmatrix}
\sim
\begin{pmatrix}
 g &  &  \\
  & a+b+g   \\
  &  & -a-b \\
\end{pmatrix}.
\label{mat:simC0C1c}
\end{align}
\end{subequations}

Note that we can write the system \eqref{sys:v} in the form 
\begin{equation}
( x I_{4} - RTR^{-1} )
\frac{dv}{dx}
=
\varLambda v,
\label{sys:vdash}
\end{equation}
where
$ 
\varLambda
= 
\begin{pmatrix}
\lambda_{++--} J &                \\
              & \lambda_{+--+} J  \\
\end{pmatrix} 
$. 
The Euler transformation
\begin{equation}
v_{\mu}(x)
=\int_{C}(t-x)^{\mu-1}v(t)\,dt,
\label{trans:Eulerv}
\end{equation}
where $C$ is an appropriate path of integration,
transforms
the system \eqref{sys:vdash} into
\begin{equation*}
( x I_{4} - RTR^{-1} )
\frac{dv_{\mu}}{dx}
=
( \varLambda + \mu I_{4} ) v_{\mu},
\end{equation*}
which
is written in the form
\begin{equation}
\frac{dv_{\mu}}{dx}
=R( x I_{4} - T )^{-1}R^{-1}(\varLambda+\mu I_{4})v_{\mu}.
\label{sys:vlambda}
\end{equation}
If $\mu=\lambda_{--++}$,
then
\begin{equation*}
\varLambda+\mu I_{4}
=\varLambda+\lambda_{--++} I_{4}
=\begin{pmatrix}
0 &   &   &   \\
  & 2\lambda_{--++} &   &   \\
  &   & 2\lambda_{0-0+} &   \\
  &   &   & 2\lambda_{-0+0} \\
\end{pmatrix},
\end{equation*}
and hence
the system \eqref{sys:vlambda} is reducible.
We write
\begin{equation}
R(xI_{4}-T)^{-1}R^{-1}(\varLambda+\lambda_{--++} I_{4})
=
\frac{1}{x}
\begin{pmatrix}
 0 & \ast \\
 0 & K_{0}
\end{pmatrix}
+
\frac{1}{x-1}
\begin{pmatrix}
 0 & \ast \\
 0 & K_{1}
\end{pmatrix},
\label{def:K0K1}
\end{equation}
where $K_{0}$, $K_{1}$ are $3\times3$ matrices,
and
\begin{equation*}
v_{\lambda_{--++}}
=
\begin{pmatrix}
 \ast \\ \tilde{v}
\end{pmatrix},
\end{equation*}
where $\tilde{v}$ is a $3$-dimensional vector.
Then $\tilde{v}$ satisfies
the system
\begin{equation}
\frac{d\tilde{v}}{dx}
=
\left(
\frac{1}{x}
K_{0}
+
\frac{1}{x-1}
K_{1}
\right)
\tilde{v}.
\label{sys:vtilde}
\end{equation}
Since the matrices in the right hand side of \eqref{def:K0K1}
satisfy
\begin{align*}
\begin{pmatrix}
 0 & \ast \\
 0 & K_{0}
\end{pmatrix}
&=
R \begin{pmatrix}
 I_{2} &  \\
   & O
\end{pmatrix} (A+\lambda_{--++} I_{4}) R^{-1}
\\
&
=
R \begin{pmatrix}
 \lambda_{++++}J+\lambda_{--++} I_{2}  &  \ast \\
   & O
\end{pmatrix} R^{-1},
\\
\begin{pmatrix}
 0 & \ast \\
 0 & K_{1}
\end{pmatrix}
&=
R \begin{pmatrix}
  O &  \\
   & I_{2}
\end{pmatrix} (A+\lambda_{--++} I_{4}) R^{-1}
\\
&
=
R \begin{pmatrix}
  O &  \\
 \ast & \lambda_{-+-+}J+\lambda_{--++} I_{2}
\end{pmatrix} R^{-1}
\end{align*}
and
\begin{equation*}
\begin{pmatrix}
 0 & \ast \\
 0 & K_{0}
\end{pmatrix}
+
\begin{pmatrix}
 0 & \ast \\
 0 & K_{1}
\end{pmatrix}
=
R (A+\lambda_{--++} I_{4}) R^{-1}
=
\varLambda+\lambda_{--++} I_{4},
\end{equation*}
we have
\begin{align*}
K_{0}
\sim
\begin{pmatrix}
 \lambda_{++++}J+\lambda_{--++} I_{2}  &   \\
     & 0 \\
\end{pmatrix}
&=
\begin{pmatrix}
\beta_{1}+\beta_{2} & & \\
 & -\alpha_{1}-\alpha_{2} & \\
   &   & 0 \\
\end{pmatrix},
\\
K_{1}
\sim
\begin{pmatrix}
 \lambda_{-+-+}J+\lambda_{--++} I_{2}  &   \\
     & 0 \\
\end{pmatrix}
&=
\begin{pmatrix}
-\alpha_{1}+\beta_{2} & & \\
 & -\alpha_{2}+\beta_{1} & \\
   &   & 0 \\
\end{pmatrix}
\end{align*}
and
\begin{align*}
K_{0}+K_{1}
&=
\begin{pmatrix}
 2\lambda_{--++} &   &   \\
   & 2\lambda_{0-0+} &   \\
   &   & 2\lambda_{-0+0} \\
\end{pmatrix}
\\
&
=
\begin{pmatrix}
-\alpha_{1}-\alpha_{2}+\beta_{1}+\beta_{2} & & \\
 & -\alpha_{2}+\beta_{2} & \\
 & & -\alpha_{1}+\beta_{1} \\
\end{pmatrix}.
\end{align*}
Comparing these with 
\eqref{mat:simC0C1a}--\eqref{mat:simC0C1c},
we set
\begin{equation}
\alpha_{1}=a, \quad
\alpha_{2}=c, \quad
\beta_{1}=-b, \quad
\beta_{2}=a+b+c+g.
\label{cond:setting}
\end{equation}

\begin{theorem}
Under the situation \eqref{cond:setting},
the change of variable
\begin{equation}
z=x^{2\lambda_{++00}}(x-1)^{2\lambda_{0+-0}}\tilde{Q}\tilde{v},
\label{change:tildev2z}
\end{equation}
where
\begin{equation*}
\tilde{Q}
=
\begin{pmatrix}
   2\lambda_{00+0}
 & \frac{2\lambda_{+000}\lambda_{0+0-}}{\lambda_{+--+}}
 & -\frac{2\lambda_{00+0}\lambda_{+0-0}\lambda_{+++-}}
         {\lambda_{+--+}\lambda_{+-++}}
 \\[\jot]
   2\lambda_{+000}
 & \frac{2\lambda_{+000}\lambda_{0+0-}}{\lambda_{+--+}}
 & -\frac{2\lambda_{+000}\lambda_{+0-0}}
         {\lambda_{+--+}}
 \\[\jot]
   2\lambda_{+0+0}
 & \frac{4\lambda_{+000}\lambda_{0+0-}}{\lambda_{+--+}}
 & -\frac{2\lambda_{+0-0}\lambda_{+++-}}{\lambda_{+--+}}
\end{pmatrix},
\end{equation*}
transforms \eqref{sys:vtilde} into \eqref{sys:DF}.
\end{theorem}

\begin{proof}
For $C_{0}+C_{1}-(a+b+2c)I_{3}$, we take 
\begin{equation}
Q=
\begin{pmatrix}
-bp & q & -b(2b+g)r \\
ap & q & -a(2a+g)r \\
(a-b)p & 2q & (2a+g)(2b+g)r
\end{pmatrix},
\label{mat:Qabg}
\end{equation}
where $p$, $q$, $r$ are non-zero constants, 
that consists of right eigenvectors of the matrix to satisfy
\begin{equation*}
Q^{-1}(C_{0}+C_{1}-(a+b+2c)I_{3})Q
=
\begin{pmatrix}
 g &  &  \\
  & a+b+g &  \\
  &  & -(a+b) \\
\end{pmatrix}.
\end{equation*}
This means that
\begin{equation*}
Q^{-1}(C_{0}+C_{1}-(a+b+2c)I_{3})Q=K_{0}+K_{1}
\end{equation*}
in the case \eqref{cond:setting}.
Moreover,
by direct calculation we can verify that
\begin{align*}
Q^{-1}(C_{0}-(a+c)I_{3})Q&=K_{0},
\\
Q^{-1}(C_{1}-(b+c)I_{3})Q&=K_{1}
\end{align*}
hold separately if
\begin{equation}
q=-\frac{a(a+b+g)}{2a+2b+g}p, \quad
r=\frac{a+b}{(2a+g)(2a+2b+g)}p.
\label{value:pqr}
\end{equation}

The matrix $\tilde{Q}$ is nothing but
\eqref{mat:Qabg} with \eqref{value:pqr} and $p=1$
represented by
$\alpha_{1}$, $\alpha_{2}$, $\beta_{1}$, $\beta_{2}$
instead of
$a$, $b$, $c$, $g$.
By the transformation
\eqref{change:tildev2z},
we have
\begin{align*}
\frac{dz}{dx}
&=
x^{2\lambda_{++00}}(x-1)^{2\lambda_{0+-0}}
\left(\frac{2\lambda_{++00}}{x}+\frac{2\lambda_{0+-0}}{x-1}\right)
\tilde{Q}\tilde{v}
\\
&\phantom{{}={}}{}
+
x^{2\lambda_{++00}}(x-1)^{2\lambda_{0+-0}}\tilde{Q}
\left(\frac{1}{x}K_{0}+\frac{1}{x-1}K_{1}\right)
\tilde{v}
\\
&=
\left\{
\left(\frac{a+c}{x}+\frac{b+c}{x-1}\right)I_{3}
+\tilde{Q}\left(\frac{1}{x}K_{0}+\frac{1}{x-1}K_{1}\right)\tilde{Q}^{-1}
\right\}
z
\\
&=
\left(\frac{1}{x}C_{0}+\frac{1}{x-1}C_{1}\right)
z.
\end{align*}
This completes the proof.
\end{proof}

Following the process stated above,
we can construct a solution of \eqref{sys:DF}.

\begin{theorem}
The system \eqref{sys:DF} has a solution of the form
\begin{align*}
z
&=
x^{a+c}(x-1)^{b+c}
\begin{pmatrix}
   -b
 & -\frac{a(a+b+g)}{2a+2b+g}
 & -\frac{b(a+b)(2b+g)}{(2a+g)(2a+2b+g)}
 \\[\jot]
   a
 & -\frac{a(a+b+g)}{2a+2b+g}
 & -\frac{a(a+b)}{2a+2b+g}
 \\[\jot]
  a-b
 & -\frac{2a(a+b+g)}{2a+2b+g}
 & \frac{(a+b)(2b+g)}{2a+2b+g}
\end{pmatrix}
\\
&\phantom{{}={}}
\cdot
\begin{pmatrix}
\displaystyle
\int_{0}^{x}
    \begin{array}[t]{@{}r@{}l@{}}
    (t-x)^{\frac{g}{2}-1}t^{a+c+\frac{g}{2}} &
    {}_{2}F_{1}(a,-b+1,a+c+\frac{g}{2}+1;t) \\
    \times &
    {}_{2}F_{1}(c,a+b+c+g+1,a+c+\frac{g}{2}+1;t)\,dt
    \end{array}
\\
\displaystyle
\int_{0}^{x}
    \begin{array}[t]{@{}r@{}l@{}}
    (t-x)^{\frac{g}{2}-1}t^{a+c+\frac{g}{2}} &
    {}_{2}F_{1}(a+1,-b,a+c+\frac{g}{2}+1;t) \\
    \times &
    {}_{2}F_{1}(c,a+b+c+g+1,a+c+\frac{g}{2}+1;t)\,dt
    \end{array}
\\
\displaystyle
\int_{0}^{x}
    \begin{array}[t]{@{}r@{}l@{}}
    (t-x)^{\frac{g}{2}-1}t^{a+c+\frac{g}{2}} &
    {}_{2}F_{1}(a,-b+1,a+c+\frac{g}{2}+1;t) \\
    \times &
    {}_{2}F_{1}(c+1,a+b+c+g,a+c+\frac{g}{2}+1;t)\,dt
    \end{array}
\end{pmatrix},
\end{align*}
where we assume that
$\Re(\frac{g}{2})>0$ and $\Re(a+c+\frac{g}{2}+1)>0$.
\end{theorem}

\begin{proof}
Under the situation \eqref{cond:setting},
we have
\begin{align*}
\gamma
&=
\frac12(\alpha_{1}+\alpha_{2}+\beta_{1}+\beta_{2})+1
=
a+c+\frac{g}{2}+1,
\\
\mu
&=
-\lambda_{++--}
=
\frac{g}{2},
\\
\tilde{Q}
&=
\begin{pmatrix}
   -b
 & -\frac{a(a+b+g)}{2a+2b+g}
 & -\frac{b(a+b)(2b+g)}{(2a+g)(2a+2b+g)}
 \\[\jot]
   a
 & -\frac{a(a+b+g)}{2a+2b+g}
 & -\frac{a(a+b)}{2a+2b+g}
 \\[\jot]
  a-b
 & -\frac{2a(a+b+g)}{2a+2b+g}
 & \frac{(a+b)(2b+g)}{2a+2b+g}
\end{pmatrix}.
\end{align*}
Combining these with
\eqref{sol:v},
\eqref{trans:Eulerv} and
\eqref{change:tildev2z},
we obtain the solution above.
\end{proof}

\section*{Acknowledgements}
The author wishes to express his thanks to
Akihito Ebisu for several helpful comments.

\end{document}